\documentclass[preprint,12pt]{elsarticle}
\usepackage{epsfig}

\usepackage{amsmath,amsthm,amssymb,latexsym}
\usepackage{amsfonts}
\usepackage[american]{babel}
\usepackage{mathrsfs,color,hyperref}
\usepackage{cleveref}

\usepackage{cleveref}

\newcommand{\pa}{\partial}
\newcommand{\R}{\mathbb{R}}
\newcommand{\RR}{{\mathbb R}}
\newcommand{\la}{\lambda}
\newcommand{\Tan}{{\mathrm{Tan}}\,}
\newcommand{\Nor}{{\mathrm{Nor}}\,}
\newcommand{\vare}{\varepsilon}

\newtheorem{theorem}{Theorem}[section]
\newtheorem{definition}[theorem]{Definition}
\newtheorem{lemma}[theorem]{Lemma}

\newtheorem{rem}[theorem]{Remark}
\newtheorem{proposition}[theorem]{Proposition}

\newtheorem{example}{Example}
\newtheorem{corollary}[theorem]{Corollary}

\newcommand{\dist}{\mbox{\rm dist}}

\newcommand{\Ga}{\Gamma}

\newcommand{\Om}{\Omega}

\newcommand{\co}{{\mathrm{co}}}

\newcommand{\inte}{{\mathrm{int}}}

\newcommand{\cl}{{\mathrm{cl}}\,}
\renewcommand{\span}{{\mathrm{span}}}

\newcommand{\rev}[1]{#1}

\journal{Advances in Applied Mathematics}

\begin{document}

\begin{frontmatter}



\title{R-hulloid of the vertices of a tetrahedron}


\author[unifi]{Marco Longinetti}
\affiliation[unifi]{organization={DIMAI, Università degli Studi di Firenze}, addressline={V.le Morgagni 67/a}, city={Firenze}, postcode={50138}, state={}, country={Italy}}

\author[limoges]{Simone Naldi}
\affiliation[limoges]{organization={Univ. of Limoges, CNRS, XLIM, UMR 7252}, addressline={123 Av. Thomas}, city={Limoges}, postcode={87060}, country={France}}
%\affiliation[sorbonne]{organization={Lip6, Sorbonne University}, addressline={4 Place Jussieu}, city={Paris}, postcode={F-75005}, country={France}}

\author[unifi2]{Adriana Venturi}
\affiliation[unifi2]{organization={Università degli Studi di Firenze}, addressline={V.le Morgagni 67/a}, city={Florence}, postcode={50138}, state={}, country={Italy}}

%% Abstract
\begin{abstract}
The $R$-hulloid, in the Euclidean space $\R^3$, of the set of vertices $V$ of a tetrahedron $T$ is the minimal closed set containing $V$ such that its complement is \rev{the union} of open balls of radius $R$. When $R$ is greater than the circumradius of $T$, the boundary of the $R$-hulloid consists of $V$ and possibly of four spherical subsets of well defined spheres of radius $R$ through the vertices of $T$. The existence of a value $R^*$ such that these subsets collapse into a point $O^*$, in the interior of $T$, is investigated; in such a case $O^*$ belongs to four spheres of radius  $R^*$, each one through three vertices of $T$ and not containing the fourth one. As a consequence, the range of $\rho$ such that $V$ is a $\rho$-body is described completely. This work generalizes to dimension three previous results, proved in the planar case and related to the three circles Johnson's Theorem.
\end{abstract}


\end{frontmatter}



\section{Introduction}  

Let $R$ be a positive real number. The $R$-hulloid of a closed set $E$ in the Euclidean space $\R^d$,
denoted by $\co_R(E)$, is the minimal closed set containing $E$ such that its complement
$(\co_R(E))^c$ is \rev{the union} of open balls of radius $R$. More generally, if the complement of a body ({\it i.e.},
non empty closed set) $E$ is \rev{the union} of open balls of radius $R$, then $E$ is called an $R$-body. This concept
is a natural generalization of the convex hull $\co(E)$ of a set $E$, in which the role of closed
half-spaces is played by the closed complement of balls of equal radius.

%In other words, $\co_R(E)$ is the intersection of all $R$-bodies containing $E$ (indeed, the family of
%$R$-bodies is closed with respect to intersection, see \cite{Per}), and it coincides with $E$ if and
%only if $E$ is itself a $R$-body.

The study of $R$-bodies and of $R$-hulloids of finite sets or more generally, finite unions
of compact sets has been studied before (for instance under the name of $\epsilon$-convex sets \cite{Per});
random approximations by such sets have been considered in set estimation theory \cite{Lop}.
These sets find application in statistics, but also in other areas such as
stereology or computational geometry (see \cite{Cue} and references therein).

There has been much work regarding different generalizations of convexity that involve, instead of
intersections of closed half-spaces, some different family $\mathscr{C}$ of subsets of the Euclidean space.
The case when $\mathscr{C}$ is a family of complements to balls was investigated in \cite{russi}.
The work \cite{Kab} describe the case when $\mathscr{C}$ is a family of affine transformations of a given
convex body. Finite intersections or unions of balls have been considered, \rev{for example,} in the
context of the Kneser-Poulsen conjecture, \rev{see \cite{Bez} and references therein.}

\rev{Finally, when $\mathscr{C}$ is a family of unit balls, one obtains a special subclass of convex bodies,
  the so-called ball-bodies. Ball-bodies have been studied extensively, see, for example, \cite{Art}
  for a recent paper, and \cite{Sch} for results on the class of floating bodies related to ball-bodies.
  A comprehensive survey and a list of references can be found in \cite{Bez}.}
%Finally, when $C$ is a family of unit balls, a special subclass of convex bodies, the ball-bodies are studied in \cite{Art} and in \cite{Sch} for the class of floating bodies.
\\

%, finds applications, among others, in molecular mechanics
%and dynamics simulations in chemistry:
%in this context, $E$ is a finite union of balls representing the atoms of a molecule embedded in a solute,
%and its $R$-hulloid is the complement of the locus where the solute has access (itself modeled as a
%sphere of some radius $R$, allowed to roll around the set $E$); in particular, the boundary of $\co_R(E)$
%is a surface enclosing the forbidden region, called the Solvent Excluded Surface, see
%\cite{lee1971interpretation,quan2016mathematical}.

{\bf Main contributions and outline}.
In this paper we consider the situation where $E \subset \R^d$ is the set of vertices $V$ of a simplex.
The planar setting ($d=2$) is by far more understood than the general $d$-dimensional case. When $E=V$ is
the set of vertices of a triangle $T$, a complete description of $\co_\rho(V)$ was proved in \cite{LMV}
as elementary consequence of the three circles Johnson's Theorem \cite{johnson,johnson2} (we recall
this result from \cite{LMV} in Proposition \ref{cor(v)inplane}). Our main goal is to give a geometric
description of $\co_\rho(V)$, for all $\rho > 0$ in dimension $d\geq 3$ and in particular for $d=3$.

The description of $\co_\rho(V)$ looks similar to the planar case using special $\rho$-supporting spheres
of $V$, of which a precise definition is given in Definition \ref{defRsupporting}. An explicit representation
of $\co_\rho(V)$ is obtained in Theorem \ref{coRdsimplex}. The proof relies on non trivial geometric arguments
that are valid in every dimension $d\geq 3$, where Johnson's Theorem cannot be applied.

The shape of $\co_\rho(V)$ depends on a special value of a radius by $r_L(V)$ defined as the supremum of the
set of $\rho$ for which $V$ is a $\rho$-body. This critical value is defined in (\ref{defrL(V)}) as the
infimum of the set of $\rho$ such that $\co_\rho(V)$ has not empty interior.

For the regular simplex $T = \co(V)$ in dimension $d \geq 2$, of circumradius $r(V)$, it is proved in
\cite[Theorem 5.6]{LMV} that $r_L(V) = \frac{d}{2}r(V)$; in such a case there exist $d+1 $ spheres of radius
$r_L(V)$, supporting $V$ and intersecting at the center of $T$. The definitions and the results of Section
\ref{cased=3} are obtained in any dimensions $d\geq3$; however the aim of the paper is to restrict to the
three dimensional case since there is a connection with a solved conjecture for unit-sphere systems in
$\R^3$, see \cite{HMNT,MT2}. This fact gives us a sharp bound for $r_L(V)$ for well-centered tetrahedra
in Theorem \ref{wellcenteredrL(V)}.

For this reason in Section \ref{4crossingpointtetrahedrum} we restrict the attention to configurations of
four spheres $S_j$ in $\R^3$, of same radius $R^L$, so that $S_j$ contains three vertices $V \setminus \{v_j\}$
of $T$, and such that the spheres intersect at a single point $O^L$. If $O^L$ does not lie on the circumsphere
of $T$, it is called a {\em four crossing point} of $T$; when $R^L=r(V)$ this configuration of spheres including
the circumsphere of $T$ would be a unit-sphere system in $\R^3$ if they are all distinct, Definition
\ref{defunitshoeresystem}, see \cite{HMNT} and Remark \ref{unitspheresystem}.

In Theorem \ref{existpL(V)wellcentered}, it is proved that there exist four special spheres $R^*$-supporting
$V$ and with a common intersection at one point $O^*\in \inte(T)$ if and only if $R^* := r_L(V) > r(V)$; moreover
in this case $\co_{R^*}(V)$ reduces to $V \cup \{O^*\}$. As a consequence of this fact, a complete description
of $\co_\rho(V)$ for all $\rho>0$ is given in Corollary \ref{finale su T well centered} for all tetrahedra $T$.

Interestingly, the fifth point $O^*$ in the hulloid might not coincide with special central points of $T$,
see Remark \ref{remarkO*notparticular}.

A condition for the uniqueness of the four-crossing radius and four crossing point is given in Theorem \ref{uniquenessRLOL}.
Conversely, examples where $R^L$ and $O^L$ are not unique are obtained in \cite{NL} using methods from symbolic computation,
for the family of triangular pyramids, see Remark \ref{remark7points}.

\section{Definitions and preliminaries on R-bodies}

$\R^d$ is the euclidean space of dimension $d\geq 2$. The elements of $\R^d$ are called vectors. The origin of $\R^d$ is denoted by $o = (0,\ldots,0)$. The open ball of center $x \in \R^d$ and radius $r>0$ is denoted by $B(x, r)$, its boundary is the sphere $\pa B(x,r)$; the unit sphere $\pa B(o,1) \subset \R^d$ is often denoted by $S^{d-1}$.
{In case  $R$ is a fixed positive real number, let us simply denote by $B(x)$ the open ball of radius $R$ and center $x$, and by $B$ any open ball of radius $R$}. A body is a closed non empty set. The closure of a set $A$ is $\cl(A)$, its interior is $\inte(A)$ and its complement $A^c$. The usual scalar product between vectors $u,v\in \RR^d$ will be denoted by $\langle u,v\rangle$. A cone $C$, with vertex $o$, is a subset of $\R^d$ with the following property: $\forall x\in C$ and $\forall \la \geq 0$, then $\la x \in C$. A cone $C$ is pointed if $C \cap (-C) = \{o\}$. The apex set of a closed cone $C$ is $C\cap (-C)$ and it contains $o$.

Let $A$ be a body and $q \in A$. The {\em tangent cone} of $A$ at $q$ is defined as:
$$
\Tan(A,q)=\{v \in \RR^d: \forall \varepsilon > 0 \, ,
\exists  x\in A\cap B(q,\vare)\; , \exists r>0 \,\mbox{s.t.} \; |r(x-q)-v|< \vare\}.$$
Let us recall that if  $\Tan(A,q)\neq \{o\}$ then 
$$S^{d-1}\cap \Tan(A,q)= \bigcap_{\vare >0}
cl\left\{\frac{x-q}{|x-q|}:\, x\in A\cap B(q,\vare) \text{ and }  x\neq q\right\}.$$
The {\em dual cone} of a cone $K$ is
$K^\star=\{y\in \RR^d : \langle y,x \rangle \geq 0, \, \forall x \in K\}.$
The {\em normal cone} at
$q$ to $A$ is the non empty closed convex cone, given by:
\begin{equation*}
  \label{defNor}
  \Nor(A,q) := \{u\in\RR^d: \langle u,v\rangle \le 0, \, \forall v \in \Tan(A,q)\} = -\{\Tan (A,q)\}^\star.
\end{equation*}

 \begin{definition}\label{d1}
   A body $A \subset \R^d$ will be called an \emph{$R$-body} if $\forall y \in A^c$ there exists an open
   ball $B  \subset \R^d$ of radius $R$ satisfying $y\in B \subset A^c$, that is if:
   $$A = \bigcap_{B \cap A = \emptyset} B^c.$$
  \end{definition}
 
  \begin{definition}\label{d2}
  Let $ E \subset \R^d$ be a body and let $R$ be a positive real number. The set
  $$
  \co_R(E) :=  \bigcap_{B \cap E = \emptyset} B^c.
  $$
  will be called the \emph{$R$-hulloid} of $ E$. 
 \end{definition}
  If the family of all open balls of radius $R$ not intersecting $E$ is empty, we assume $\co_R(E) = \R^d$.
  The $R$-hulloid always exists and $\co_R(E)$ is the minimal $R$-body  containing $E$. 
  Clearly every convex body $E$ is an $R$-body (for all positive $R$), in such a case $E=\co(E)=\co_R(E)$
  for every $R$.
  A body $A$ is an $R$-body if and only if $A = \co_R(A)$.
  Moreover, see \cite[formula (8)]{Per}:
  \begin{equation}\label{monotonicity}
  R_1 \leq R_2 \Rightarrow \co_{R_1}(E) \subset \co_{R_2}(E).
  \end{equation}

  \begin{proposition}\cite[Thm~3.10]{LMV}\label{r5}
    The closed subsets of a sphere of radius $r$ are $R$-bodies for all $R\leq r$.
  \end{proposition}
Let us look at special sets $E=V$, where  $V$ is the set of  vertices of a simplex $T=\co(V)$ in $\R^d$.

\begin{definition}\label{defcircumradius}
  The following fact is well known: there exists a unique open ball
  $B(c(V),r(V))$ such that $V \subset \pa B(c(V),r(V))$; it is 
  called the \emph{circumball} to $T=\co(V)$. Its radius $r(V)$ is the \emph{circumradius} of $T$
  and its center is the \emph{circumcenter} $c(V)$, denoted below simply by $c$.
  Let us recall that its closure $D(V) := \cl(B(c,r(V)))$ does not coincide (in general) with the
  closed ball of minimum radius containing $V$.
  \end{definition}
  %By definition of circumball of a tetrahedron $T$, the set of vertices $V$ lies on a sphere $ \Sigma=\pa  $ of radius $r(V)$, the circumradius of $T$.

  The following proposition is a consequence of Proposition \ref{r5}, with $r=r(V)$ and $R=\rho$:

  \begin{proposition}\label{corollarycoR(V)}
    Let $T = \co(V)$ be a simplex in $\R^d$. Then
    %\begin{equation*}
    %  \label{r(V)>=R}
      $\co_\rho(V)=V$ for all $\rho \leq  r(V)$.
    %\end{equation*}
  \end{proposition}  
  \begin{definition}(\cite{LMVnew}) \label{defRsupporting} Let $A$ be a closed set in $\R^d$ and $a\in \pa A$.
    Let $v\in \partial B(o,1)=S^{d-1}$.
    We say that the ball $B(a+Rv)$ (of radius $R$) is \emph{$R$-supporting $A$ at $a$} if 
 $A\subset (B(a+Rv))^c$. The sphere $\pa B(a+Rv)$ is called \emph{$R$-supporting $A$ at $a$}.
 \end{definition}

  \subsection{$R$-cones}

  \begin{definition}\label{defconvonS}
    Let $\rho > 0$ and $S = \partial B(o,\rho)$.
    There is a one-to-one correspondence between closed cones $K \subset \RR^d$ and closed subsets
    $\mathcal{K} \subset S \subset \R^d$:
    for any closed cone $K$ of $\RR^d$, let $\mathcal{K} = K \cap S$;
    conversely, for any closed set $\mathcal{K} \subset S$, let
    $K = \{\la v: v\in \mathcal{K}, \la \geq 0 \}$ be the related cone in $\RR^d$.
   \end{definition}
 
  In the previous definition $\partial B(o,\rho)$ can be replaced by a sphere $S$ centered at any
  point $p \in \R^d$ and $K$ by a cone with vertex $p$.
  Let us recall a family of $R$-bodies, called $R$-cones, see \cite[\S 5]{LMV}.
 The $R$-cones are a generalization of convex cones, as the $R$-bodies are a generalization of convex sets.

\begin{definition}\label{curvedconedef}
 Let $\mathcal{K}$ be a body in  $S^{d-1}$. An \emph{$R$-cone with vertex $o$} (or simply \emph{$R$-cone}) is the $R$-body: 
 \begin{equation*}
  \label{defunboundcone}
   C_{\mathcal{K}}:=\bigcap_{v \in \mathcal{K}} (B(Rv,R))^c.
 \end{equation*} 
Similarly, for $x\in \RR^d$, an $R$-cone with vertex $x$ is the $R$-body:
$C_{\mathcal{K}}^x:=x+ C_{\mathcal{K}}=\bigcap_{ v\in \mathcal{K}} (B(x+Rv,R))^c.$
\end{definition}

\begin{proposition}\cite[Theorem 5.6]{LMV}\label{viceversaregular} Let $\mathcal{K}$ be a body in  $S^{d-1}$ 
and $K$ its  related cone in $\RR^d$. Then
 \begin{enumerate}
 \item[$a$)] $\Tan(C_{\mathcal{K}}, o) = -K^\star$;
 %\item[$b$)]  $\Nor(C_{\mathcal{K}})\cap S^{d-1} = \co(\mathcal{K}) \cap S^{d-1}$;
 \item[$b$)] $ \Tan(C_{\mathcal{K}})\setminus \{o\}\subset \inte(C_{\mathcal{K}})$.
 \end{enumerate} 
\end{proposition}

\begin{lemma}\label{trepalle}
  Let $B_j \subset \R^d, j=1,\ldots,d$, be open balls of radius $\rho$, not necessarily distinct. Let $p^* \in \cap_{j=1}^d \pa B_j$; then for every neighborhood $\mathcal{U} $ of $p^*$, one has 
$\inte\large(\mathcal{U}\setminus \cup_{j=1}^d B_j\large) \neq \emptyset.$
\end{lemma}
 \begin{proof} Let  $u_j$ be the unit inner normal at $p^*$ to $B_j$. Let
  $\mathcal{K}= \{u_1,u_2,\ldots, u_d\}\subset S^{d-1}$. The closed set $ \cap_{j=1}^d B_j^c$ is called $\rho$-cone with vertex $p^*$ in \cite{LMVnew}. The $\rho$-cone $C_\mathcal{K}=\cap_{j=1}^d B_j^c-p^*$ with vertex at $\{o\}$ is a generalization of the usual convex cone. Let $K=\{\la u_i \,|\, u_i\in \mathcal{K}, \lambda \geq 0\}$ be the related cone to $\mathcal{K}$. By properties $a$) and $b$) of Proposition \ref{viceversaregular}
  \begin{equation}\label{Tan(C_K)inC_K}
   -K^*-\{o\}=\Tan(C_{\mathcal{K}})-\{o\} \subset \inte(C_\mathcal{K})
  \end{equation}
 where $\Tan(C_{\mathcal{K}})$ is the tangent cone to the $\rho$-cone $C_\mathcal{K}$ at $o$. Since $\mathcal{K}$ is contained in a closed hemisphere, the dual cone $K^*$ is not empty and the tangent cone is non empty too. Then, by inclusion \eqref{Tan(C_K)inC_K}, for every $v\neq o$ in the tangent cone:  $v\in \inte(C_\mathcal{K})$. Let $y=p^*+v$, then  $y\in \inte((\cup_{j=1}^d B_j)^c)\neq \emptyset$ and 
 \begin{equation}\label{dist(y,C_K)>0}
 \dist (y,  \cup_{j=1}^d B_j) > 0;
 \end{equation}
  the thesis follows.
 \end{proof}

 \begin{lemma}\label{trepallenelvertice} Let $B_j \subset \R^d, j=1,\ldots,d$, be open balls of radius $\rho$. Let $p^* \in \cap_{j=1}^d \pa B_j$; let $C_\mathcal{K}^{p^*}=\cap_{j=1}^d B_j^c$ be the related $\rho$-cone with vertex $p^*$ and  let  $\mathcal{T}$ be a closed convex 
    cone with vertex $p^*$.  Then 
   \begin{itemize}
   \item[$a$)] if $\Tan(\mathcal{T}, p^*)\cap \Tan(C_\mathcal{K}^{p^*},p^*)\setminus \{o\}\neq \emptyset $
     then,  for every  neighborhood $\mathcal{U} $ of $p^*$, there exists a point $y\in \mathcal{T}\cap \mathcal{U}$ such that
      \begin{equation}\label{dist(y,C_K)c>0}
     \dist(y, (C_\mathcal{K}^{p^*})^c) > 0; 
 \end{equation}

\item[b)]  if $\Tan(\mathcal{T}, p^*)\cap \Tan(C_\mathcal{K}^{p^*},p^*)= \{o\}$, then  
 $p^*$ is an isolated point of $\mathcal{T}\cap C_\mathcal{K}^{p^*}$.
 \end{itemize}
\end{lemma}
\begin{proof} It can be assumed that $p^*=o$. In case $a$), let $r$ be a ray starting from $o$ in the cone $\mathcal{T}\cap C_\mathcal{K}$. From \eqref{Tan(C_K)inC_K} for every $y\in r, y\neq o$, the inequality \eqref{dist(y,C_K)>0} holds. Then  for every neighborhood $\mathcal{U} $ of $o$ and for every $y \in r\cap\mathcal{U}, y\neq o $ formula \eqref{dist(y,C_K)c>0} holds.

In case b),  by contradiction let $y_n \to o$, a sequence of points $y_n \in 
\mathcal{T}\cap C_\mathcal{K}$, $y_n\neq o$. Then up to a subsequence 
$v=\lim_n y_n/|y_n|\neq o $ belongs both to $\Tan(\mathcal{T}) = \mathcal{T}$ and
to $\Tan( C_\mathcal{K})$; this is a contradiction.
\end{proof}

\subsection{Planar case}
 Let us recall the following result as consequence  of Johnson's Theorem 
  \cite{johnson} (cf. fig.~\ref{fig1}).

%% COMMENTATO
  \begin{figure}[!ht]
    \begin{center}
      \includegraphics[scale=1]{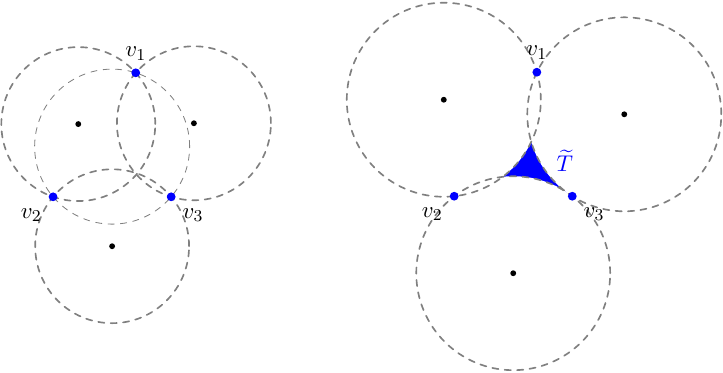}
    \end{center}
    \caption{$\rho$-hulloid of three points in $\R^2$, $\rho=r(V)$ (left) and $\rho>r(V)$ (right).}
    \label{fig1}
  \end{figure}
    
  \begin{proposition}\cite[Theorem 4.2]{LMV}\label{cor(v)inplane}
    Let $V \subset \R^2$ be the set of vertices of a triangle $T$ with circumradius $r(V)$.
    If $\rho >r(V)$, then $$\co_\rho(V)=V\cup \tilde{T},$$
    where  $\tilde{T}\subset T $ is the curvilinear triangle  bordered by three arcs of circles of radius $\rho$, each one through two vertices of $T$. If $T$ is right-angled or obtuse-angled then the vertex of the major angle of $T$ is also a vertex of $\tilde{T}$.
  \end{proposition}

\section{$R$-hulloids of the vertices of a simplex}\label{cased=3}

Let $V = \{v_i, i=1, \ldots, d+1\} \subset \R^d$ be the set of vertices of a simplex $T=\co(V) \subset \R^d$.
%%%A description of $\co_\rho(V)$ is  the formula \eqref{repcorho}. 
%%%Since Johnson's Theorem does not hold in $\R^3$, the approach is different. 

Let us notice that, by Proposition \ref{corollarycoR(V)}, $\co_\rho(V) = V$ for $\rho \leq r(V)$.

From now on, let us assume $\rho > r(V)$ and let us denote by $c=c(V)$ the circumcenter of $T$, so that $B(c,r(V))$ is the circumball of $T$. 

Let us denote by $V_i=V\setminus \{v_i\}$, for $1\leq i \leq d+1$, and by $T_i=\co(V_i)$
the facets of $T$, by $H_i = \span(T_i)$ and let $c_i,r_i$ be the circumcenter and the circumradius respectively of $T_i$ in $H_i$.

\begin{definition} The simplex $T=co(V)$ will be called \emph{well-centered} if $c(V) \in \inte(T)$ (in  dimension two, the well-centered simplexes are the acute-angled triangles).
\end{definition}

Let $D(V)$ be the closed circumball of $T$. The set $D(V) \cap H_i$ is a  closed circle of radius $r_i$ and center $c_i$. Let $\Om_i$  be its relative interior. 
 
As the vertices in $V$ are affinely independent, then $v_i \not\in H_i$, for all $i$.
Let $H_{i+}$ be the open half space bounded by $H_i$ containing $v_i$, $\mathsf{H}_i^+ = \cl(H_{i+})$ and $l_i$ the line orthogonal to $H_i$ containing $c_i$.

%Let $l_i$ be the line orthogonal to $H_i$ at $c_i$ and let $n_i$ be the unit   exterior normal vector to   $H_i^+$, let $o_i\in l_{i+}$ be  the half line:  
%  \begin{equation}\label{ninotcentered}
%_{i+}= \{c+\la n_i, \la > 0\}\subset l_i.
%\end{equation}

\begin{lemma} \label{lemmagenerale}
  For every $\rho > r(V)$ and for any face $T_i$ of $T=\co(V)$ and line
  $l_i$ orthogonal to $H_i$ at the circumcenter $c_i$ of $T_i$, there
  exists a unique point $o_i(\rho)$ such that
 $V_i \subset \pa B(o_i(\rho), \rho)\cap H_i$
 and
 $v_i\not \in B(o_i(\rho),\rho).$
 Moreover
 $\rho \mapsto \mathsf{H}_i^+   \cap (B(o_i(\rho), \rho))^c$ is a continuous map of strictly nested $\rho$-bodies for $\rho > r(V)$.
 \end{lemma}
\begin{proof}
 Let us denote $B_i(\rho) := B(o_i(\rho),\rho)$,
    $B(V) := \inte(D(V)) = B(c,r(V))$ and
    $B_i^+(r(V)) := \lim_{\rho \to r(V)^+} B_i(\rho)$.
% $B_i(\rho):\equiv B(o_i(\rho), \rho)$, for $\rho > r(V)$, and 
 %\begin{equation*}  \label{defBi(rho(V))}
 %\end{equation*}
%\begin{equation*}\label{defBi(rho(V)+)}
%  \begin{aligned}
%    B(V) & := \inte(D(V)) = B(c,r(V)) \\[.5em]
%    B_i^+(r(V)) & := \lim_{\rho \to r(V)^+}B_i(\rho).
%  \end{aligned}
%\end{equation*}
Let us notice that two different cases can occur:
\begin{itemize}
\item[$a$)] $T$ is  well-centered. In this case $\forall i$, $c\in H_{i+}$; $o_i(\rho)$ lies on the half
  line $l_i^*$ of $l_i$, with origin in $\tilde{c_i}$, the symmetric point of $c$ with respect to $H_i$,
  not intersecting $H_i^+$. The half line $l_i^*$ is oriented as the outwards unit vectors $n_i$ to the
  facet $T_i$ of $T$. In this case $B(V) \neq B_i^+(r(V)) = B(\tilde{c_i},r(V))$ and 
 $$\lim_{\rho \to r(V)^+} \mathsf{H}_i^+\cap \big(B_i(\rho)\big)^c = \mathsf{H}_i^+\cap  \big(B_i^+(r(V)\big)^c;$$
\item[$b$)] $T$ is not well-centered: then, there are indices $i$ such that the plane $H_i$ separates $c$ from $v_i$;
  for these indices,  $o_i(\rho) $ lies on the half line $l_i^*$ of $l_i$ starting from $c$ and not  intersecting
  $H_i^+$; the half line $l_i^*$ is again oriented as the outwards unit vectors $n_i$ to the facet $T_i$ of $T$.
  For such indices $i$ it turns out that $B(V) = B_i^+(r(V))$, then
  $$\lim_{\rho \to r(V)^+}    \mathsf{H}_i^+ \cap (B_i(\rho))^c  = \mathsf{H}_i^+\cap (B(V))^c.$$
\end{itemize}
%The proof now follows from elementary geometric arguments.
\end{proof}

\begin{rem}
$$ \bigcap_i \Big(\mathsf{H}_i^+ \cap (B(c,r(V)))^c\Big)= \Big(\bigcap_i \mathsf{H}_i^+\Big)  \bigcap \Big(B(c,r(V))\Big)^c=  T\cap \Big(B(c,r(V))\Big)^c=V.$$
\end{rem}

 \begin{lemma}\label{lemamcentrofuori}
If $\rho > r(V)$,
  then every open ball $B(z, \rho)$, not intersecting $V$ and intersecting $\co(V)$, has the following properties:
\begin{itemize}
\item[a)] $\dist(z,c) > \rho-r(V);$ 
\item[b)] the radical plane $H$ through $\pa B(c,r(V))\cap\pa B(z,\rho)$
separates $z$ from $T$;
\item[c)] $\dist(z, T) > (\rho-r(V))/2$.
\end{itemize}
  \end{lemma}
   \begin{proof} Since $\inte(T)\subset B(c,r(V))$, the spheres $\pa B(c,r(V))$ and $\pa B(z,\rho)$ have non empty intersection; thus they have a real radical plane $H$.   Since every vertex $v_i$ belongs to $\pa B(c,r(V))\setminus B(z,\rho) $ then
    $\dist (z,c)> \dist (z,v_i)-\dist(v_i,c)> \rho-r(V).$
     Moreover $H$  separates $D(V)\setminus B(z,\rho) $  from $z$, then it separates $V$ from $z$,  this proves $b$). Since $\rho > r(V)$, then $H$ has distance from $c$ less than from $z$, therefore
   $\dist(z,T) > \dist(z,H) > \dist (z, (c+z)/2) >  (\rho-r(V))/2$, 
   and $c$) is proved.
   \end{proof}

   \begin{theorem}\label{coRdsimplex}
     Let $T=\co(V)$ be a simplex in $\R^d$ and let $\rho > r(V)$,
     and $B_i(\rho)$ as in Lemma \ref{lemmagenerale}.
     Then 
 \begin{equation}\label{repcorho}
 \co_\rho(V)= \co(V)\setminus \left( \bigcup_{i =1}^{d+1} B_i(\rho) \right)
 \end{equation}
  and $ r(V) < \rho \mapsto \co_\rho (V)$ is a not decreasing family of closed sets,  with
 \begin{equation}\label{rhotoinfty}
 \lim_{\rho\to +\infty}\co_\rho(V)=V\cup \inte(T).
 \end{equation}
    \end{theorem}  
\begin{proof}
  First let us notice that  
  $$ G:= \co(V)\setminus \left( \bigcup_{i =1}^{d+1} B_i(\rho) \right)= \bigcap_i \mathsf{H}_i^+ \bigcap_i B_i^c(\rho).$$
  Then $G$ is a $\rho$-body as intersection of $\rho$-bodies, and $G \supset V$, thus 
  $F:= \co_\rho(V) \subset G.$
Let us prove now that $G\subset F$. This is equivalent to prove that
\begin{equation}\label{FcomplinGcompl}
F^c \subset G^c.
\end{equation}

First  $V \subset F$, by definition of $\rho$-hulloid.
Let $y \not \in F$, then $y\not \in V$. If $y\not \in T$, then it is obvious that $y \not\in G$; if $y \in \pa T\setminus V$,  there exists a face $T_i$  such that $y\in T_i\setminus V \subset B_i(\rho)$. Then, $y\in G^c$.

If $y \in \inte(T)$ and $y\not \in F$, then by definition of hulloid there exists an open ball
$\tilde{B}(z,\rho) \ni y$, with $\tilde{B}(z,\rho)\cap V=\emptyset$. Let us consider the family
$$\Ga_y=\{\tilde{B}(x,\rho): \tilde{B}(x,\rho)\ni y, \tilde{B}(x,\rho)\cap V=\emptyset\}.$$
Since $\Ga_y$ is not empty let $$r_y=\inf_z \{|z-y|: y\in B(z,\rho)\in \Ga_y \} .$$ Obviously $r_y< \rho$.
 From Lemma \ref{lemamcentrofuori}, $r_y > 0$;  by continuity argument, $r_y$ is a minimum and there exists $B(z^*, \rho) \in \Ga_y$ such that 
 $|z^*-y|=r_y.$ Assume for the moment that $d=3$. \\
 
 {\em Claim A}:  $\pa B( z^*, \rho)\cap V \neq \emptyset$.
 
 By contradiction if $\pa B( z^*, \rho)\cap V = \emptyset$, moving $z^*$ towards $y$, it is
 possible to get another ball $\tilde{B}(\tilde{x},\rho)\in \Ga_y$  such that 
 $$|\tilde{x}-y| < |z^*-y|= r_y,$$
 which is impossible. Then Claim A is proved and there exists $v_3\in \pa B(z^*,\rho)\cap V$. 
 With a similar argument there exists at least another vertex $v_2 \neq v_3$, $v_2 \in \pa B(z^*, \rho)$, otherwise a rotation of
 $B(z^*,\rho)$ around $v_3$ towards $y$ would decrease the value of $r_y$, which would contradict the definition of $r_y$.
 
 \vspace{0.2cm} 
 
 {\em Claim B}: $\{v_1,v_2,v_3\}\subset  \pa B( z^*, \rho) $.

 First let us notice that, since  $y\in \inte(T)$, it does not belong to the line through $v_2,v_3$; moreover $|z^*-v_3|=|z^*-v_2|$.
 Let $\mathcal{C}$ be the circle with center $(v_2+v_3)/2$, through $z^*$, in the plane orthogonal to the axis $v_2v_3$. Let  $\Lambda$ be the plane through $y,v_2,v_3$ and let $z^+$ be the closest point of $\mathcal{C}\cap \Lambda $ to $y$.

The  function 
 $z \to |z-y|$, such that  $ y\in B(z,\rho)\in \Ga_y $, has minimum at $z^*$, then its restriction to $\mathcal{C}$ has the same minimum value at $z^*$. Since the function $\mathcal{C}\ni z \to |z-y|$, with no  restriction on $z$, attains its minimum on $z^+$ and decreases its value moving $z$ on $\mathcal{C}$ towards $z^*$; then two cases need to be considered:
 \begin{itemize}
 \item[$a$)] $z^*=z^+$;
 \item[$b$)] there exists $v_1\in \pa B(z^*,\rho)\cap V$.
  \end{itemize}
 Case $a$) cannot hold. Indeed, by item $c$) of Lemma \ref{lemamcentrofuori}, the radical plane of the two spheres $\pa B(z^+,\rho)$, $\pa B(c,r(V))$ (containing $v_2,v_3$) separates $z^+$ from $y$.
 
 Then item $b$) holds: in such a case $B(z^*, \rho)$ cannot rotate towards $z^+$, to get  another $B(z,\rho) \in \Ga_y, B(z,\rho)\ni y $, to decrease further the value of $r_y$.  Then  $y $ belongs to the open ball $B(z^*,\rho)$, which coincides with $B_4(\rho)$ of 
 the family of balls defined in Lemma \ref{lemmagenerale}. Therefore $y\in  G^c$. This proves \eqref{FcomplinGcompl} for $d=3$.

A similar argument proves that \eqref{FcomplinGcompl} holds for any $d>3$ and \eqref{repcorho} is proved.
 
 The monotone property  of $\co_\rho(V)$ follows from \eqref{monotonicity}.  
Since, see \cite{Per},
 \begin{equation*}\label{limintcorho}
 \bigcup_{\rho > r(V)}\inte(\co_\rho(V)) = \inte(T),
 \end{equation*}
 then the limit property  \eqref{rhotoinfty} holds.
\end{proof}

\begin{corollary}
Let $B_i(\rho)$ and $B_i^+(r(V))$ be as in Lemma \ref{lemmagenerale}.
If $T$ is not well-centered, then
 \begin{equation}\label{limrhoco(V)r(V)+notcentered}
   \lim_{\rho \to r(V)^+}\co_\rho(V)=V ;
  \end{equation}
  if $T$ is  well-centered, then
  \begin{equation}\label{limrhoco(V)r(V)+centered}
    \lim_{\rho \to r(V)^+}\co(V)\setminus \bigcup_{i =1}^{d+1}  B_i(\rho) =\co(V)\setminus
    \bigcup_{i =1}^{d+1} B_i^+(r(V)).
  \end{equation}
 \end{corollary}
 \begin{proof} In case $T$ is not well-centered, then there exists $i$ such that 
 $H_i$ separates $c$ from $v_i$, then 
 $$\lim_{ \rho \to r(V)^+}B_i(\rho)=B(c,r(V))\supset \co(V)\setminus V$$
 and \eqref{limrhoco(V)r(V)+notcentered} holds; the limit \eqref{limrhoco(V)r(V)+centered} follows from continuity and monotone properties of the map $\rho \mapsto \co_\rho(V)$.\end{proof}

\section{Critical R-hulloid of the vertices of a tetrahedron}

In this section we focus on the three-dimensional case since there is a connection with
a solved conjecture about unit-sphere-systems in $\R^3$ \cite{HMNT,MT2}.
Let $V$ be the set of vertices of a tetrahedron $T = \co(V) \subset \R^3$.
%We consider configurations of distinct supporting spheres of $V$ of equal radius
%which intersect at a single point of $T$.

%describe particular configurations of spheres defining the $\rho$-hulloid of $V$.
%We assume indeed that the $\rho$-supporting spheres intersect at a single point of $\inte(T)$.
First let us consider a more general question considered also in \cite{NL}.

\subsection{The four-crossing point of a tetrahedron}\label{4crossingpointtetrahedrum}
\begin{definition}\label{def4-crossing}
  Given a tetrahedron $T$ of vertices $V = \{v_i, 1 \leq i \leq 4\}$,
  let $R^L$ be the radius of four distinct spheres $S_j, 1\leq j\leq 4$, satisfying
  the following conditions, see fig.~\ref{fig2}:
  \begin{itemize}
  \item[$a$)] $S_j$ contains the vertices $v_i$, for all $i \neq j$;
  \item[$b$)] the intersection of the four spheres is one point $\{O^L\}$;
  \item[$c$)] $O^L $ does not belong to the circumsphere of $T$.
\end{itemize}  
  If there exist $R^L$ and $O^L$ satisfying the previous conditions, they are called the
  \emph{four-crossing radius} and the \emph{four-crossing point of $T$}, respectively.
\end{definition}

\begin{figure}[!ht]
  \begin{center}
    \includegraphics[scale=1]{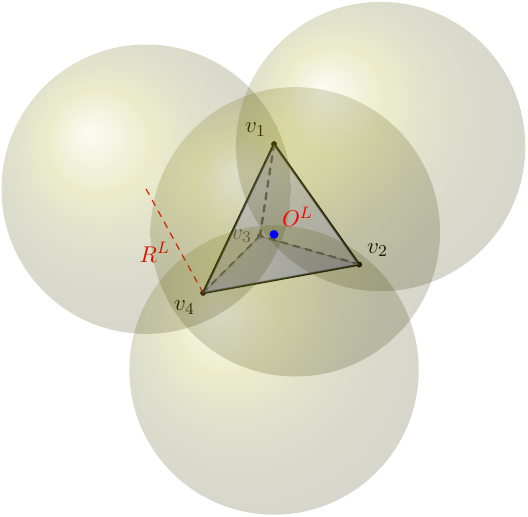}
  \end{center}
  \caption{A four-crossing radius and a four-crossing point of the regular tetrahedron.}
  \label{fig2}
\end{figure}

%In the plane, the definition of three-crossing radius and three-crossing point for a triangle can be given in a similar way. In this case, from Johnson' Theorem one has $R^L = r(V)$, the circumradius of $T = \co(V)$, and $O^L$ is the orthocenter of $T$. In such a case, condition $c$) of Definition \ref{def4-crossing} holds, except for the right triangle for which $O^L$ is a vertex. Then a uniqueness and existence result holds for three-crossing radius and point for all triangles, except for rectangular triangles.

Note that the spheres $S_j$ in Definition \ref{def4-crossing} are not necessarily supporting spheres of $V$.
The method developed in \cite{NL}, based on symbolic computation, allows
to determine the values of $R^L$ and $O^L$ for the family of triangular pyramids.

The following example describes a special pyramid in $\R^3$ with a configuration
of supporting spheres similar to that of the rectangular triangle in the plane.

\begin{example}[{\cite[Example 5]{NL}}]
  \label{example6}
  Let $T = \co(V)$ be the triangular pyramid with apex $v_1$, equidistant 1 from the vertices
  $V \setminus \{v_1\}$ of an equilateral base $T_1 = \co(V \setminus \{v_1\})$, with sides
  length $\sqrt{12/5}$.
  Then there exists a unique set of four distinct spheres $S_j$ of radius $\rho = r(V) = \sqrt{5/4}$,
  the circumradius of $T$, where $S_1$ is the circumsphere of $V$ and $S_2,S_3,S_4$ satisfy the
  following properties:
  \begin{enumerate}
  \item
    $S_j$ contains the vertices $v_i$, for all $i \neq j$;
  \item
    $\cap_{j=2}^4 S_j = \{v_1\}$.
  \end{enumerate}
\end{example}

Let us notice that in this case property $c$) of Definition \ref{def4-crossing} is not satisfied
and therefore, this configuration does not correspond to a four-crossing point of $T$.
The existence of the four-crossing radius $R^L$ and the four-crossing point $O^L$ has some
connections with the unit-sphere-systems in $\R^3$.

\begin{definition}[{\cite{HMNT}}]
  \label{defunitshoeresystem}
  A \emph{unit-sphere-system in $\R^d$} is a set of $d+2$ distinct spheres in $\R^d$
  of equal radius, with empty intersection, and such
  that each subgroup of $d+1$ spheres has a common intersection.
\end{definition}

\begin{rem}\label{unitspheresystem}
  It was proved in \cite[Thm. 8.1]{MT2} that there are no unit-sphere-systems for $d=3$.
  This implies that $R^L$ does not coincide with the circumradius $r(V)$ of a tetrahedron $T=\co(V)$:
  indeed, if this was the case, the set of four spheres $S_j$ supporting $V$, together with the
  circumsphere of $T$, would form a unit-sphere-system in $\R^3$.
\end{rem}

\subsection{$R^*$-hulloid of the vertices of a tetrahedron}

\begin{definition}\label{defhull}
  Let $\rho > 0$. The $\rho$-hulloid of $V$ is called \emph{full} if $\inte(\co_\rho(V)) \neq \emptyset$.
\end{definition}

\begin{definition}\label{defrL(V)} Let us define
  $$
  r_L(V) = \inf \{\rho > r(V): \co_\rho(V) \text{ is full}\}.
  $$ 
\end{definition}
Let us notice that, for $d=2$ and $\rho>0$, Propositions \ref{corollarycoR(V)} and \ref{cor(v)inplane}
  imply that the set $\co_\rho(V)$ is full if and only if $\rho > r(V)$.

  \begin{lemma}\label{remarkdefoi}
    For $\rho > r(V)$, the points $o_i(\rho)$, $1 \leq i \leq 4$,
    defined by Lemma \ref{lemmagenerale}, are the vertices of a simplex $W$.
     \end{lemma} 
 \begin{proof} The points $o_i(\rho)$ are distinct. Otherwise if $o_i=o_j$ then $B_i(\rho)=B_j(\rho)$ would be an open ball of radius $\rho$ containing all the vertices $V$ on the boundary in contradiction with assumption $r( V) < \rho$. Moreover, see cases $a$) and $b$) in the proof of Lemma \ref {lemmagenerale},  by construction, each point $o_i$ belongs to the corresponding  half  lines $l_i^*$ contained in the lines $l_i$ through the circumcenter $c$ and orthogonal to the facet $T_i$. All  distances of $o_i(\rho)$ from $c$ are  positive; the
 vectors from $c$ to $c_i$ are outwards normal vectors to the facet $T_i$ of $T$. Then the simplex $W$, with  vertex $o_i(\rho)$, contains, up to a translation and dilatation,   the simplex $T'$ whose vertices are the contact points of the inscribed sphere to $T$; see also \cite{Brand-krizek} for this simplicial vertex-normal duality. 
\end{proof}

 \begin{definition}  Let us denote  $r(\rho)$ the circumradius of the simplex $W$ with
   circumcenter $c(W)$ and vertices $o_i(\rho)$, defined in Lemma \ref{remarkdefoi}.   
 \end{definition}
 
 Let us notice that, by definition of $c(W)$ and $r(\rho)$,
 the intersection of the spherical surfaces $\pa B(o_i(\rho), r(\rho))$ is $c(W)$.
 %%%; these spheres  contain the vertices $V\setminus \{v_i\}$ if $\rho= r(\rho)$.

 \begin{theorem}\label{existpL(V)wellcentered} Let $V$ be the set of vertices of a tetrahedron $T$. 
   There exist $R^* > r(V) $ and a configuration of four distinct $R^*$-supporting spheres of $V$
   intersecting at some $O^*\in \inte(T)$  if and only if  
   \begin{equation}\label{R*>r(V)}
    r_L(V)> r(V).
   \end{equation} In such a case $R^*=r_L(V)$ and
   \begin{equation}\label{repcorl(V)}
      \co_{R^*}(V)=V\cup \{O^*\}.
   \end{equation} 

      \end{theorem}
 \begin{proof}
   If \eqref{R*>r(V)} holds, let us consider, for all $\rho > r(V)$, the centers $o_i(\rho)$  of the balls $B_i(\rho)$,
   $i=1, \ldots,4$, as in \eqref{repcorho}; let $r=r(\rho)$ be the circumradius of $W = \co\{o_1(\rho), \ldots,\allowbreak o_4(\rho)\}$,
   and let $c(W)$ be its circumcenter. Then for every $i$, the map $\rho \mapsto r(\rho) = |c(W) - o_i(\rho)|$ is continuous.

   Let $R^* = r_L(V)$. Then one has
   \begin{equation*}
     %%\label{limrho*+}\label{limcorho*+}
     r(R^*)=\lim_{\rho \to R^*} r(\rho) \,\,\,\,\,\,\,\, \text{ and } \,\,\,\,\,\,\,\, \co_{R^*}(V)=\lim_{\rho \to (R^*)^+} \co_\rho(V).
   \end{equation*}
   Let us prove now that 
   \begin{equation}\label{limrho*-}
     r(R^*)= R^*.
   \end{equation}
 The set 
  $T\setminus \cup_{i =1}^{4} B_i(\rho) $ has non empty interior for $\rho > R^*$, and it is contained in a bounded set. Therefore there exists a sequence  $\rho_n \to (R^*)^+$ and  a point  $p^*$ such that 
  \begin{equation}\label{p*=lim}
  p^*\in \lim_n \inte(T\setminus \cup_{i =1}^{4} B_i(\rho_n)).
  \end{equation}

Three cases could happen:
 \begin{itemize}
  \item[$i$)] $p^*\in \inte(T)$;
  \item[$ii$)] $p^*\in \pa T\setminus V$;
   \item[$iii$)] $p^*\in V$;
\end{itemize}
 
First, let us consider case $i$). Then there exists a neighborhood $\mathcal{U} \subset\inte(T)$ of $p^*$. In case   $p^* \in \pa B_i(R^*)$ for every $i$, then  $p^*$ is the circumcenter $O^*$ of  $W$ and  \eqref{limrho*-} holds.
Otherwise, by contradiction, if  all four spheres do not cross at $p^*$, then  there exists $i$ and  at most three spheres $\pa B_j(R^*), j\neq i$, intersecting at $p^*$. Then, by Lemma \ref{trepalle}, $\tilde{\mathcal{U}}=\mathcal{U}\cap \cap_{j\neq i }(B_j(R^*))^c$ has non empty interior. Since $\dist(p^*,B_i(R^*)) > 0$    then  $ B_i(R^*)^c \cap \tilde{\mathcal{U}}$ has  non empty interior. Then $T\setminus \cup_{j =1}^{4} B_j(R^*)$ has non empty interior too. By continuity, for $\varepsilon > 0$ small enough,  $T\setminus \cup_{j =1}^{4} B_j(R^*_\varepsilon)$ has non empty interior, with $R^*_\varepsilon=R^*-\varepsilon < R^*$. Then $\co_{R^*_\varepsilon}(V)$ is full, in contradiction with the definition of $r_L(V)=R^*$.

Let us prove now that case $ii$) does not occur.
Since  the distance of  $T\setminus \cup_{i =1}^{4} B_i(\rho_n) $ from every $o_i(\rho_n)$ is $\rho_n$, then, by \eqref{p*=lim},
 $\dist( p^*, o_i(R^*)) \geq  R^*, \quad \forall i.$ Obviously $p^*\in T$. Since for every $q$ on the face $T_i\setminus V$,
 $\dist (q ,o_i(R^*)) < R^*$, then $p^*$ does not belong to $T_i\setminus V$, for every $i$.

Let us prove  that case $iii$) does not occur.
By contradiction let  $p^*=v_4 \in V$. If  $p^*\in \pa B_4(R^*)$ then  $\pa B_4(R^*)$ would be the circumsphere of $T$ and $R^*=r(V)$, in contradiction with the assumption \eqref{R*>r(V)}. Then  $p^*\not \in \pa B_4(R^*)$, therefore 
$\dist (p^*, B_4(R^*)) > 0$. 

Let us consider the $R^*$-hulloid of $V$, namely $\co_{R^*}(V)= T \cap _{i =1}^{4} (B_i(R^*))^c.$ Since $\dist(p^*, B_4(R^*)) > 0$, there
exists a neighbourhood $\mathcal{U}$ of $p^*$, such that $$\co_{R^*}(V)\cap \mathcal{U}= \mathcal{U} \cap
 \big(T \cap _{i =1}^{3} B_i(R^*)^c)=  \mathcal{U} \cap \big( T \cap (C_\mathcal{K}^{p^*})\big),$$ 
where  $C_\mathcal{K}^{p^*}$ is the $R^*$-cone of Lemma \ref{trepallenelvertice}, with $\rho=R^*$.

Let us consider Lemma \ref{trepallenelvertice} with  $\mathcal{T}:=\Tan (T, p^*)$.

In case $a$) of Lemma \ref{trepallenelvertice}, for every neighborhood $\mathcal{U}$ of $p^*$, small enough, there exists $y\in \mathcal{T}\cap  \mathcal{U} \subset T$, $y\in \inte( C_\mathcal{K}^{p^*}) $.
By continuity argument there exists $ R_\varepsilon$,  $r(V) < R_\varepsilon < R^*$, such that
$y\in \inte\cap_{i =1}^{4} (B_i(R_\varepsilon))^c.$ Since  every neighborhood of $y\in T$ has non empty intersection with $\inte(T)$,
then $$\inte(T)\cap\inte\big( \cap_{i =1}^{4} (B_i(R_\varepsilon))^c\big) \neq \emptyset.$$
Then $\co_{R_\varepsilon}(V)=T\cap_{i =1}^{4} (B_i(R_\varepsilon))^c$ has non empty interior and it is full. This is impossible
by definition of $r_L(V)=R^*$.

In case $b$) of Lemma \ref{trepallenelvertice} the vertex $p^*$ is an isolated point of
$T\setminus \cup_{i =1}^{4} B_i(R^*)=\co_{R^*}(V)$. This is in contradiction with \eqref{p*=lim}.

% This is impossible since 
% Sia $u\neq 0 $ un elemento di tale cono limite. Allora $v_1+\varepsilon u$ per $\varepsilon >0$ sufficientemente piccolo e' fuori di $B_1(R^*)$ ed  e' interno a $\co_{R^*}(V)$. Cioe' la $\dist(p^*+\varepsilon u, B_j(R^*)) > 0$ per ogni $j$. Ma allora considerando $R_\delta=R^*-\delta < R^*$ per delta sufficientemente piccolo resta ancora 
%$$\dist(p^*+\varepsilon u, B_j(R_\delta)) > 0.$$
%Quindi ancora $p^*+\varepsilon u $ stat nell'interno di $\co_{R_\delta}(V)$. Cio' e' assurdo per la definizione di inf di $R^*=r_L(V)$, vedi definizione  \ref{defrL(V)}.

Then, only case $i$) holds and $p^*=O^*\in \inte(T)$ is the circumcenter of $W$. Let us prove now \eqref{repcorl(V)}.
 
Since $O^* \in \big( \bigcup_{i =1}^{4} B_i(R^*) \big)^c \cap T,$ by \eqref{repcorho}, it follows that $O^*\in \co_{R^*}(V)$: then
\begin{equation}\label{O*incoR(v)}
V\cup \{O^*\} \subset \co_{R^*}(V). 
\end{equation} 
Let us consider the pyramid  $T_i^*$ with apex $O^*$ and opposite face $T_i$. As $O^*\in \pa B(o_i(R^*), R^*)$, then 
 $ T_i^* \subset B(o_i(R^*), R^*)\cup V\cup \{O^*\};$
 thus
 $$\co_{R^*}(V)=T\setminus \bigcup_i B(o_i(R^*), R^*) \subset \left( T\setminus \bigcup T_i^*\right)\cup V \cup \{O^*\}=V\cup \{O^*\}.$$
 This and \eqref{O*incoR(v)} prove \eqref{repcorl(V)} under the condition \eqref{R*>r(V)}.
 
 In case \eqref{R*>r(V)} does not hold, that is if $r_L(V) = r(V)$, then a configuration of four supporting spheres $S_i$
 of radius $r(V)$ intersecting at a point $O^* \in \inte(T)$, $O^*\not\in S$, the circumsphere of $T$, cannot exist.
 By contradiction the collection $\{S , S_i, i=1,\ldots,4\}$ would be a unit-sphere-system in $\R^3$. This is in contradiction
 with \cite[Thm.8.1]{MT2}, see Remark \ref{unitspheresystem}.
 \end{proof}
 
 \begin{corollary}\label{finale su T well centered}
   Let $V$ be  the set of vertices of a tetrahedron $T=T(V)$ in $\R^3$,  and  let $R^*=r_{L}(V)$.  Then 
   \begin{eqnarray*}\label{rho<rLV} 
     \co_\rho(V) = & V  & \mbox{\quad for \quad }  \rho < R^* ; \\
     \label{co_R(V)=}
     \co_\rho(V)  =& V\cup \tilde{\Gamma} & \mbox{\quad for \quad }   R^* < \rho,
   \end{eqnarray*}
   where  
   \begin{equation*}\label{W=}
     \tilde{\Gamma}= \inte(T)\setminus  \bigcup_i B_i(\rho)
   \end{equation*}
   is  a connected, not empty set, 
   with $\pa \tilde{\Gamma}$ \rev{the union} of connected subsets of $\pa B_i(\rho)$.
   
For $\rho=R^*$ it holds
$$\co_{R^*}(V)= V \cup \{O^*\} \mbox{\quad if \quad } r_L(V) > r(V),$$
$$\co_{R^*}(V)= V \mbox{\quad if \quad } r_L(V) = r(V).$$ 
 \end{corollary}
  
 \begin{theorem}\label{wellcenteredrL(V)}If $T$ is well-centered then \eqref{R*>r(V)} and \eqref{repcorl(V)} hold. 
 \end{theorem}
 \begin{proof}
   By contradiction, if \eqref{R*>r(V)} does not hold, then  $r_L(V)=r(V)$ and for all $\rho_n > r(V)$ the set 
 $T\setminus \cup_{i =1}^{4} B_i(\rho)$ has non empty interior.
   Arguing as in the proof of Theorem \ref{existpL(V)wellcentered}, there exists a point $p^*$ satisfying \eqref{p*=lim}. Since  \eqref{R*>r(V)} does not hold, Theorem \ref{existpL(V)wellcentered} implies that $i)$ can not occur; the same argument in the proof of Theorem \ref{existpL(V)wellcentered} proves that case $ii)$ cannot occur; since $T$ is well-centered, every vertex $v_i$ has a positive distance greater than $\delta(T) > 0$ from all balls $B(o_i(\rho_n), \rho_n)$. Then, $p^*$ cannot be a vertex of $V$. Therefore $iii)$ cannot occur, too.
   This is impossible.  \end{proof}  

 \begin{rem}\label{remarkO*notparticular}
   For the trirectangular pyramid, see \cite[Example 4]{NL},  $O^*$ differs in general from the
   orthocenter, and hence from the Monge point, of some  triangular pyramid. It differs also from
   the circumcenter. In this example, condition \eqref{R*>r(V)} holds; therefore since the trirectangular pyramid is not well-centered this implies that condition
   \eqref{R*>r(V)} is not equivalent to the condition that $T$ is well-centered. Note that in the plane the two conditions are equivalent.
\end{rem}

 Let us notice that the value $R^*=r_L(V)$, satisfying \eqref{R*>r(V)} and the point $O^*$, satisfying \eqref{repcorl(V)} in Theorem
\ref{existpL(V)wellcentered}, are a four-crossing radius and a four-crossing point respectively
of $T$, since  $O^* \in \inte(T)\subset B(V)$.
%%%; under this condition on $O^L$ the following uniqueness result is obtained.

\begin{theorem}\label{uniquenessRLOL}
  Let $(R^L,O^L)$ be a four-crossing radius and a four-crossing point for a simplex
  $T = \co(V)$. If $O^L\in \inte(T)$, then $R^L$ and $O^L$ are  uniquely determined.
\end{theorem}
\begin{proof}
  %Since $O^L \in \inte(T)$ every sphere $S_j=\pa B_j$ containing $V\setminus \{v_j\}$ and $O^L$ is uniquely determined.
  Since $O^L\in \inte(T)$ then $v_j\not \in \cl(B_j)$; by Lemma \ref{lemmagenerale} the family 
  $\rho \mapsto \mathsf{H}_j^+   \cap (B(o_j(\rho), \rho))^c$
  is strictly nested. Then,  $B_j$ has radius $R^L$  greater than $r(V)$.  Therefore $S_j$ is the boundary
  of an open ball $B_j$, which is $R^L$-supporting $V$ at all vertices $v_i$ with $i\neq j$.
  Then $R^L=r_L(V)$ of  Definition \ref{defrL(V)} and satisfies  \eqref{R*>r(V)} of Theorem \ref{existpL(V)wellcentered}. This implies that $R^L=R^*$ and $O^L=O^*$ defined by \eqref{repcorl(V)}.
\end{proof}

\begin{rem}\label{remark7points}
  Without the condition $O^L\in \inte(T)$, the four-crossing radius and point of a tetrahedron
  are not uniquely determined, as proved in \cite{NL} for the class of triangular pyramids.
  For a regular tetrahedron there are seven different four-crossing points, one of which
  is the center of the tetrahedron, see Figure \ref{fig2} and \cite[Example 1]{NL}.
\end{rem}

\section*{Conclusions}
The $\rho$-hulloid $\co_\rho(E)$ of a body $E$ is a generalization of the convex hull of $E$, in which
the support hyperplanes of $E$ (resp. open half spaces not intersecting $E$), important tools of convex
analysis, are replaced by spheres of equal radius $\rho$ supporting $E$ (resp. by open balls not
intersecting $E$).

This paper provides a representation formula of $\co_\rho(V)$ for $V$ the set of vertices of a tetrahedron
$T \subset \R^3$. This result generalizes a similar formula in $\R^2$ that derives from the three circles
Johnson’s Theorem \cite{johnson}, that cannot be applied in dimension $d>2$. It is explicitly proved here
in $\R^3$ for connections with some applications (for instance in graphical simulation of molecular models
\cite{quan2016mathematical} or in image analysis \cite{Cue}) and with a recent result on the non-existence
of unit-sphere-systems in $\R^3$ \cite{HMNT,MT2}.

Based on continuity arguments on the family of spheres
of radius $\rho$ supporting $V$, the construction of four spheres of radius $R^*$, supporting $V$ and
intersecting at a point $O^*$, is obtained for all tetrahedra $T$.

In this paper it is also proved that for every $T$, there exists $R^*$, greater or equal to the circumradius
of $T$, such that the set $\co_\rho(V)$ has non empty interior if and only if $\rho > R^*$; for $\rho > R^*$
the boundary of $\co_\rho(V)$ consists of four spherical subsets of four spheres of radius $R^*$.
For $\rho = R^*$ the set $\co_\rho(V) \setminus V$ collapses in a point $O^*$ of $T$. The value of $R^*$ cannot
attain the value of the circumradius of $T$ when $T$ is a well-centered tetrahedron. This is related to the above-mentioned
results \cite{HMNT,MT2} and to explicit algebraic computations of $R^*$ and $O^*$ obtained for the class of
triangular pyramids in \cite{NL}.

\paragraph{Acknowledgments}
This work has been partially supported by IN\-DAM-GN\-AM\-PA (\rev{2025}) \rev{and INDAM-GNSAGA (2025)}.
The second author is supported by the ANR Project ANR-21-CE48-0006-01 “HYPERSPACE”.

\end{document}